\documentclass[10pt,article]{amsart}
\usepackage{amsfonts}
\usepackage{indentfirst,amsthm,bm,amsmath}
\usepackage{hyperref}
\makeatletter
\date{}
\newtheorem{theorem}{Theorem}[section]

\newtheorem{lemma}{Lemma}[section]

\newtheorem{remark}{Remark}[section]

\begin{document}
\title[Uniqueness problems of meromorphic mappings]{Multiple values and uniqueness problem of meromorphic mappings
sharing hypersurfaces}
\author[Tingbin Cao]{Tingbin Cao}
\address[Tingbin Cao]{Department of Mathematics, Nanchang University, Nanchang, Jiangxi 330031, China} \email{tbcao@ncu.edu.cn(Corresponding author)}
\thanks{This first author was supported in part by the NSFC (no.11461042) and the NSF of Jiangxi Provence in China (no. 20161BAB201007).}
\author[Hongzhe Cao]{Hongzhe Cao}
\address[Hongzhe Cao]{Department of Mathematics, Nanchang
University, Nanchang, Jiangxi 330031, China}
\email{hongzhecao@126.com}
\thanks{This second author was supported in part by the NSFC (no.11401291).}
\date{}

\begin{abstract}The purpose of this article is to deal with the multiple values and uniqueness problem
of meromorphic mappings from $\mathbb{C}^{m}$ into the complex
projective space $\mathbb{P}^{n}(\mathbb{C})$ sharing fixed and moving hypersurfaces.
We obtain several uniqueness theorems which extend some
known results.\end{abstract} 
\maketitle

\section{Introduction and main results}
In 1926, R. Nevanlinna \cite{nevanlinna} proved the well-known
five-value theorem that for two nonconstant meromorphic functions
$f$ and $g$ on the complex plane $\mathbb{C},$ if they have the same
inverse images (ignoring multiplicities) for five distinct values in
$\mathbb{P}^{1}(\mathbb{C}),$ then $f=g.$ We know that
the number five of distinct values in Nevanlinna's five-value
theorem cannot be reduced to four. For example, $f(z)=e^{z}$ and
$g(z)=e^{-z}$ share four values $0, 1, -1, \infty$ (ignoring
multiplicities), but $f(z)\not\equiv g(z).$ \par

Fujimoto \cite{fujimoto-1} generalized the Nevanlinna's well-known five-value theorem to the case of meromorphic mappings from $\mathbb{C}^{m}$ into $\mathbb{P}^{n}(\mathbb{C}),$ and obtained that for two linearly nondegenerate meromorphic mappings $f, g$ of $\mathbb{C}^{m}$ into $\mathbb{P}^{n}(\mathbb{C}),$ if they have the same inverse images of $3n+2$ hyperplanes counted with multiplicities in $\mathbb{P}^{n}(\mathbb{C})$ in general position, then $f=g.$\par

In 1983, Smiley \cite{smiley} considered meromorphic mappings which share $3n+2$ hyperplanes of $\mathbb{P}^{n}(\mathbb{C})$ without counting multiplicity and proved the following result.\par\vskip 6pt

\begin{theorem} \cite{smiley} \label{TA} Let $f, g$ be linearly nondegenerate meromorphic mappings of $\mathbb{C}^{m}$ into $\mathbb{P}^{n}(\mathbb{C}).$ Let
${\{H_{j}}\}_{j=1}^{q}(q\geq 3n+2)$ be hyperplanes in $\mathbb{P}^{n}(\mathbb{C})$ in generate position. Assume that\\
$(a) \dim (f^{-1}(H_{i})\cap f^{-1} (H_{j}))\leq m-2$ for all $1\leq i<j\leq q,$ \\
$(b) f(z)=g(z)$ on $\cup_{j=1}^{q}f^{-1}(H_{j}),$\\
$(c) f^{-1}(H_{j})=g^{-1}(H_{j})$ for all $ 1\leq j\leq q.$\\
Then $f=g.$\end{theorem}

It is natural to extend this result to meromorphic mappings of $\mathbb{C}^{m}$ into $\mathbb{P}^{n}(\mathbb{C})$ sharing hypersurfaces. As far as we known, Dulock and Ru firstly studied this  topic and obtained the following theorem.\par

\begin{theorem} \cite{dulock-ru} \label{TC} Let ${\{Q_{j}}\}_{j=1}^{q}$ be hypersurfaces of degree $\deg Q_{j}=d_{j}$ in $\mathbb{P}^{n}(\mathbb{C})$ in general position.  Let $d_{0}=\min\{d_{1}, \ldots, d_{q}\},$ $d$ be the least common multiple of $d_{j}$'s, i. e., $d=lcm(d_{1}, \ldots, d_{q}),$ and let $M=2d[2^{n-1}(n+1)nd(d+1)]^{n}.$ Suppose that $f, g$ are algebraically nondegenerate meromorphic mappings of $\mathbb{C}^{m}$ into $\mathbb{P}^{n}(\mathbb{C})$ such that \\
$(a) \dim (f^{-1}(Q_{i})\cap f^{-1} (Q_{j}))\leq m-2$ for all $1\leq i<j\leq q,$ \\
$(b) f(z)=g(z)$ on $\cup_{j=1}^{q}(f^{-1}(Q_{j})\cup g^{-1}(Q_{j})).$\\  If
$$q>(n+1)+\frac{2M}{d_{0}}+\frac{1}{2},$$
then $f=g.$\end{theorem}

The number of hypersurfaces in the result of Dulock-Ru's result is too big. In \cite{quang-2}, Quang improved Theorem \ref{TC} by firstly giving the uniqueness theorems for meromorphic mappings sharing hypersurfaces with an accepted number of hypersurfaces. However, the number of hypersurfaces in \cite{quang-2} is still very large according to Nevanlinna's five-value theorem, and the Fujimoto's uniqueness theorem and Theorem \ref{TA}.\par

Let $V$ be a complex projective subvariety of $\mathbb{P}^{n}(\mathbb{C})$ of dimension $k$ $(k\leq n).$ Let $d$ be a positive integer. We denote by $I(V)$  the ideal of homogeneous polynomials in $\mathbb{C}[x_{0}, \ldots, x_{n}]$ defining $V,$ $H_{d}$ the vector space consisting of all homogeneous polynomials in $\mathbb{C}[x_{0}, \ldots, x_{n}]$ of degree $d.$ Define $$I_{d}(V):=\frac{H_{d}}{I(V)\cap H_{d}}\quad \mbox{and}\quad H_{V}(d):=\dim I_{d}(V).$$ Then $H_{V}(d)$ is called Hilbert function of $V.$ Each element of $I_{V}(d)$ which is an equivalent class of an element $Q\in H_{d},$ will denoted by $[Q].$\par

Let $f$ be a meromorphic mapping of $\mathbb{C}^{m}$ into $V.$ We say that $f$ is degenerate over $I_{d}(V)$ if there is $[Q]\in I_{d}(V)\setminus\{0\}$ so that $Q(f)\equiv 0,$ otherwise we say that $f$ is nondegenerate  over $I_{d}(V).$ This implies that if $f$ is algebraically nondegenerate then $f$ is nondegenerate over $I_{d}(V)$ for every $d\geq 1.$\par

The family of hypersurfaces $\{Q_{j}\}_{j=1}^{q}$ is said to be in $N-$subgeneral position with respect to $V$ if for any $1\leq i_{1}<\cdots<i_{N+1},$
$$V\cap(\cap_{j=1}^{N+1}Q_{i_{j}})=\emptyset.$$ If $\{Q_{j}\}_{j=1}^{q}$ is said to be in $n-$subgeneral position with respect to $V,$ then we say that it is in general position with respect to $V.$\par

Recently, \cite{quang-an-2014} improved Theorem \ref{TC}, the uniqueness theorems in \cite{quang-2, phuong-1}, and obtained the following uniqueness theorem for meromorphic mappings which share hypersurfaces of $\mathbb{P}^{n}(\mathbb{C})$ without counting multiplicity. Now, the number of hypersurfaces is very small because Quang and An's result implies that $q\geq 3n + 2$ in the case of meromorphic mapping sharing hyperplanes in general position, and thus it generalizes Theorem \ref{TA}.\par

\begin{theorem} \cite{quang-an-2014} \label{TD} Let $V$ be a complex projective subvariety of $\mathbb{P}^{n}(\mathbb{C})$ of dimension $k$ $(k\leq n).$ Let
${\{Q_{j}}\}_{j=1}^{q}$ be hypersurfaces in $\mathbb{P}^{n}(\mathbb{C})$ in $N-$subgeneral position with respective to $V,$ $\deg Q_{j}=d_{j}$ $(1\leq j\leq q).$ Let $d$ be the least common multiple of $d_{j}$'s, i. e., $d=lcm(d_{1}, \ldots, d_{q}).$ Let $f, g$ be meromorphic mappings of $\mathbb{C}^{m}$ into $V$ which are nondegenerate over $I_{d}(V).$ Assume that\\
$(a) \dim (f^{-1}(Q_{i})\cap f^{-1} (Q_{j}))\leq m-2$ for all $1\leq i<j\leq q,$ \\
$(b) f(z)=g(z)$ on $\cup_{j=1}^{q}(f^{-1}(Q_{j})\cup g^{-1}(Q_{j})).$\\  If
$$q>\frac{2(H_{V}(d)-1)}{d}+\frac{(2N-k+1)H_{V}(d)}{k+1},$$
then $f=g.$\end{theorem}

Quang \cite{quang-3} also proposed a uniqueness theorem for meromorphic mappings sharing slowly moving hypersurfaces without counting multiplicity, which improved the uniqueness result due to Dethloff and Tan \cite{dethloff-tan-2} and \cite{phuong}.\par

\begin{theorem}\cite{quang-3} \label{TE} Let $f$ and $g$ be nonconstant meromorphic mappings of $\mathbb{C}^{m}$ into $\mathbb{P}^{n}(\mathbb{C}).$ Let
${\{Q_{j}}\}_{j=1}^{q}$ be set of slowly (with respect to $f$ and $g$) moving hypersurfaces in $\mathbb{P}^{n}(\mathbb{C})$ in weakly general position with degree $\deg Q_{j}=d_{j}.$ Put $d$ be the least common multiple of $d_{1}, \ldots, d_{n+2}$, i. e., $d=lcm(d_{1}, \ldots, d_{n+2}),$ and $N=\left(  \begin{array}{c} n+d \\ n \\  \end{array} \right)-1.$ Let $k$ $(1\leq k\leq n)$ be an integer. Assume that \\
$(a) \dim (\cap_{j=0}^{k}f^{-1}(Q_{i_{j}}))\leq m-2$ for every $1\leq i_{0}<\cdots <i_k\leq q,$ \\
$(b) f(z)=g(z)$ on $\cup_{j=1}^{q}(f^{-1}(Q_{j})\cup g^{-1}(Q_{j})).$\\ Then the following assertions hold:\\
(i) If $q>\frac{2kN(nN+n+1)}{d}$ then $f=g.$\\
(ii)In addition to the assumptions (a) and (b), we assume further that both $f$ and $g$ are algebraically nondegenerate over $\tilde{\mathcal{K}}_{\{Q_{j}\}_{j=1}^{q}}.$ If
$q>\frac{2kN(N+2)}{d}$ then $f=g.$
\end{theorem}

It is interesting to considering multiple values and uniqueness problem for meromorphic mappings. For example, H. X. Yi
(\cite[Theorem 3.15]{yi-yang}) adopted the method of dealing with
multiple values due to L. Yang \cite{yang} and obtained a uniqueness
theorem of meromorphic functions of one variable, which generalized the famous Nevanlinna's five-value theorem:

\begin{theorem}\cite[Theorem 3.15]{yi-yang} \label{T-0} Let $f$ and
$g$ be two nonconstant meromorphic functions on $\mathbb{C},$
let $a_{j}$ $(j=1,2,\ldots,q)$ be $q$ distinct complex elements in
$\mathbb{P}^{1}(\mathbb{C})$ and take
$m_{j}\in\mathbb{Z}^{+}\cup\{\infty\}$ $(j=1,2,\ldots,q)$ satisfying
$m_{1}\geq m_{2}\geq\cdots\geq m_{q}$ and $\nu_{f-a_{j}, \leq
m_{j}}^{1}=\nu_{g-a_{j}, \leq m_{j}}^{1}$ $(j=1,2,\ldots,q).$ If
$\sum_{j=3}^{q}\frac{m_{j}}{m_{j}+1}>2,$ then $f(z)\equiv
g(z).$\end{theorem}

Later, Hu, Li and
Yang extended this result to meromorphic functions in several
variables (see \cite[Theorem 3.9]{hu-li-yang}). In 2000, Aihara\cite{aihara} generalized Theorem \ref{T-0} to the case of meromorphic mappings sharing hyperplanes from $\mathbb{C}^{m}$ into $\mathbb{P}^{n}(\mathbb{C}).$ Recently, Cao-Yi\cite{cao-yi}, L\"{u}\cite{lu}, Tu-Wang\cite{tu-wang}, Quang\cite{quang}, Cao-Liu-Cao\cite{cao-liu-cao} continued to investigate this topic.\par

The main purpose of this paper is to consider multiple values and uniqueness problem for meromorphic mappings sharing fixed or moving hypersurfaces from $\mathbb{C}^{m}$ into $\mathbb{P}^{n}(\mathbb{C}),$ and generalize Theorem \ref{TD} and Theorem \ref{TE} respectively as follows.\par

\begin{theorem} \label{T1} Let $V$ be a complex projective subvariety of $\mathbb{P}^{n}(\mathbb{C})$ of dimension $k$ $(k\leq n).$ Let
${\{Q_{j}}\}_{j=1}^{q}$ be hypersurfaces in $\mathbb{P}^{n}(\mathbb{C})$ in $N-$subgeneral position with respective to $V,$ $\deg Q_{j}=d_{j}$ $(1\leq j\leq q).$ Let $d$ be the least common multiple of $d_{j}$'s, i. e., $d=lcm(d_{1}, \ldots, d_{q}).$ Let $f, g$ be meromorphic mappings of $\mathbb{C}^{m}$ into $V$ which are nondegenerate over $I_{d}(V).$
Let $m_{j}(j=1,\cdots,q)$ be positive integers or $\infty$ with $m_{1}\geq\cdots\geq m_{q}\geq n.$ Assume that\\
$(a) \dim (f^{-1}(Q_{i})\cap f^{-1} (Q_{j}))\leq m-2$ for all $1\leq i<j\leq q,$ \\
$(b) f(z)=g(z)$ on $$\cup_{j=1}^{q}(\{z\in\mathbb{C}^{m}:0<\nu_{Q_{j}(f)}(z)\leq m_{j}\}\cup\{z\in\mathbb{C}^{m}:0<\nu_{Q_{j}(g)}(z)\leq m_{j}\}).$$ If
$$\sum_{j=3}^{q}\frac{m_{j}}{m_{j}+1}>q+\frac{d(2N-k+1)H_{V}(d)}{(k+1)(H_{V}(d)-1)}-\frac{qd}{H_{V}(d)-1}-\frac{4-2H_{V}(d)}{m_{2}+1},$$
then $f=g.$
\end{theorem}

\begin{remark} (a). Obviously, Theorem \ref{TD} is just the special case when $m_{1}=m_{2}=\cdots=m_{q}=\infty$ in Theorem \ref{T1}.\\
(b). In the case of mapping into $\mathbb{P}^{n}(\mathbb{C})$ sharing hyperplanes in general position, that means $V=\mathbb{P}^{n}(\mathbb{C}),$ $d_{i}=1$$(1\leq i\leq q),$  $H_{V}(d)=n+1,$ $N=n=k,$ then the condition $$\sum_{j=3}^{q}\frac{m_{j}}{m_{j}+1}>q+\frac{d(2N-k+1)H_{V}(d)}{(k+1)(H_{V}(d)-1)}-\frac{qd}{H_{V}(d)-1}-\frac{4-2H_{V}(d)}{m_{2}+1}$$ becomes $$\sum_{j=3}^{q}\frac{m_{j}}{m_{j}+1}>\frac{(n-1)q}{n}+\frac{n+1}{n}+\frac{2(n-1)}{m_{2}+1}$$ which implies that this is a slight improvement of \cite[Theorem 1.4]{cao-yi}. Furthermore, for the special case when $n=1,$ the condition reduces to $\sum_{j=3}^{q}\frac{m_{j}}{m_{j}+1}>2$ which is just the condition of Theorem \ref{T-0}. Hence, the above theorem generalizes Theorem\ref{T-0} and  \cite[Theorem 3.9]{hu-li-yang}.\\
(c). Both Theorem \ref{TD} and Theorem \ref{T1} imply that $q\geq 3n+2$ in the case of mapping into $\mathbb{P}^{n}(\mathbb{C})$ sharing hyperplanes in general position. However,  Chen and Yan \cite{chen-yan} obtained that $q\geq 2n+3$ in this case. Thus, there maybe exist a better condition than the conditions on the number $q$ in Theorem \ref{TD} and Theorem \ref{T1}. \end{remark}\par

\begin{theorem}\label{T2}Let $f$ and $g$ be nonconstant meromorphic mappings of $\mathbb{C}^{m}$ into $\mathbb{P}^{n}(\mathbb{C}).$ Let
${\{Q_{j}}\}_{j=1}^{q}$ be set of slowly (with respect to $f$ and $g$) moving hypersurfaces in $\mathbb{P}^{n}(\mathbb{C})$ in weakly general position with degree $\deg Q_{j}=d_{j}.$ Put $d$ be the least common multiple of $d_{1}, \ldots, d_{n+2}$, i. e., $d=lcm(d_{1}, \ldots, d_{n+2}),$ and $N=\left(   \begin{array}{c}   n+d \\    n \\  \end{array}   \right)-1.$ Let $k$ $(1\leq k\leq n)$ be an integer. Let $m_{j}(j=1,\cdots,q)$ be positive integers or $\infty$ with $m_{1}\geq\cdots\geq m_{q}\geq n.$ Assume that \\
$(a) \dim (\cap_{j=0}^{k}f^{-1}(Q_{i_{j}}))\leq m-2$ for every $1\leq i_{0}<\cdots <i_k\leq q,$ \\
$(b) f(z)=g(z)$ on $$\cup_{j=1}^{q}(\{z\in\mathbb{C}^{m}:0<\nu_{Q_{j}(f)}(z)\leq m_{j}\}\cup\{z\in\mathbb{C}^{m}:0<\nu_{Q_{j}(g)}(z)\leq m_{j}\}).$$\\ Then the following assertions hold:\\
(i) If $\sum_{j=3}^{q}\frac{m_{j}}{m_{j}+1}>q+2k-2-\frac{qd}{(nN+n+1)N}-\frac{2kN-2}{m_{2}+1},$ then $f=g.$\\
(ii)In addition to the assumptions (a) and (b), we assume further that both $f$ and $g$ are algebraically nondegenerate over $\tilde{\mathcal{K}}_{\{Q_{j}\}_{j=1}^{q}}.$ If
$\sum_{j=3}^{q}\frac{m_{j}}{m_{j}+1}>q+2k-2-\frac{qd}{(N+2)N}-\frac{2kN-2}{m_{2}+1}),$ then $f=g.$
\end{theorem}

It is easy to see that Theorem \ref{TE} is just the special case when $m_{1}=m_{2}=\cdots=m_{q}=\infty$ in Theorem \ref{T2}.

\section{Preliminaries}
\subsection{}
Set $\|z\|=(|z_{1}|^{2}+\cdots+|z_{m}|^{2})$ for $z=(z_{1},\cdots,z_{m}),$  for $r>0,$ define
\begin{equation*}
B_{m}(r):={\{z\in\mathbb{C}^{m}:\|z\|<r}\},\quad
S_{m}(r):={\{z\in\mathbb{C}^{m}:\|z\|=r}\}. \end{equation*}
Let $d=\partial+\overline{\partial}, \quad d^{c}=(4\pi\sqrt{-1})^{-1}(\partial+\overline{\partial}).$ Write
\begin{equation*}
\sigma_{m}(z):=(dd^{c}\|z\|^{2})^{m-1},\quad
 \eta_{m}(z):=d^{c}\log\|z\|^{2}\wedge(dd^{c}\|z\|^{2})^{m-1}\end{equation*}
  for $z\in\mathbb{C}^{m}\setminus{\{0}\}.$

\subsection{}
Let $\varphi(\not\equiv 0)$ be an entire function on $\mathbb{C}^{m}.$ For $a\in\mathbb{C}^{m},$ we write $\varphi(z)=\sum_{i=0}^{\infty}P_{i}(z-a),$ where the term $P_{i}(z)$ is a homogeneous polynomial of degree $i.$ We denote the zero-multiplicity of $\varphi$ at $a$ by $\nu_{\varphi}(a)=\min{\{i:P_{i}\neq 0}\}.$ Set $|\nu_{\varphi}|:=\overline{{\{z\in\mathbb{C}^{m}:\nu_{\varphi}(z)\neq 0}\}},$ which is a purely $(m-1)$-dimensional analytic subset or empty set.\par
Let $h$ be a nonzero meromorphic function on $\mathbb{C}^{m}.$ For each $a\in\mathbb{C}^{m},$ we choose nonzero holomorphic functions $h_{0},h_{1}$ on
a neighborhood $U$ of $a$ such that $h=h_{0}:h_{1}$ on $U$ and $\dim(h_{0}^{-1}(0)\cap h_{1}^{-1}(0))\leq m-2,$ we define $\nu_{h}:=\nu_{h_{0}},\nu_{h}^{\infty}:=\nu_{h_{1}},$
which are independent of the choice of $h_{0},h_{1}.$

\subsection{}
Let $f:\mathbb{C}^{m}\rightarrow\mathbb{P}^{n}(\mathbb{C})$ be a nonconstant meromorphic mapping, $(\omega_{0}:\cdots:\omega_{n})$ be an arbitrarily fixed homogeneous coordinates on $\mathbb{P}^{n}(\mathbb{C}).$ We choose holomorphic functions $f_{0},\cdots,f_{n}$ on $\mathbb{C}^{m}$ such that
$$I_{f}:={\{z\in\mathbb{C}^{m}:f_{0}(z)=\cdots=f_{n}(z)=0}\}$$ is of dimension  $\leq m-2,$  and $f=(f_{0}:\cdots:f_{n})$
is called a reduced representation of $f.$ Set $\|f\|=(\sum_{j=0}^{n}|f_{j}|^{2})^{\frac{1}{2}}.$ The characteristic function of  $f$ is defined by
\begin{equation*}
T_{f}(r)=\int_{S_{m}(r)}\log\|f\|\eta_{m}-
 \int_{S_{m}(1)}\log\|f\|\eta_{m}\quad(r>1).\end{equation*}
Note that $T_{f}(r)$ is independent of the choice of the reduced representation of $f.$\par

\subsection{}
For a divisor $\nu$ on $\mathbb{C}^{m}$ and let $k,M$ be positive integers or $\infty,$ we define the following counting functions of $\nu$ by:
\begin{equation*}
\nu^{[M]}(z)=\min\{\nu(z), M\},\quad \nu_{\leq k}^{[M]}(z)=
\left\{\begin{array}{ll}
 0, & \hbox{if}\quad \nu(z)>k; \\
 \nu^{[M]}(z), & \hbox{if}\quad \nu(z)\leq k,
\end{array}
\right.
\end{equation*}
\begin{eqnarray*} n(t)=\left\{
                                    \begin{array}{ll}
                                      \int_{|\nu|\cap B(t)}\nu(z)\sigma_{m}, & \hbox{if $m\geq 2;$} \\
                                      \sum_{|z|\leq t}\nu(z), & \hbox{if $m=1.$}
                                    \end{array}
                                  \right.
\end{eqnarray*} Similarly, we define $n^{[M]}(t),$ $n_{> k}^{[M]}(t)$ and $n_{\leq k}^{[M]}(t).$\par

Define
\begin{equation*}
N(r, \nu)=\int_{1}^{r}\frac{n(t)}{t^{2m-1}}dt\quad
(1<r<\infty).\end{equation*} Similarly, we define $N(r, \nu^{[M]}),$ $N(r,
\nu^{[M]}_{> k})$ and $N(r, \nu^{[M]}_{\leq k})$ and denote them by
$N^{[M]}(r, \nu),$ $N^{[M]}_{> k}(r, \nu)$ and $N^{[M]}_{\leq k}(r,
\nu)$ respectively.\par

For a meromorphic function $f$ on $\mathbb{C}^{m},$ we denote
by $$N_{f}(r)=N(r, \nu_{f}),\quad
N^{[M]}_{f}(r)=N^{[M]}(r, \nu_{f}),$$
$$ N^{[M]}_{f, \leq
k}(r)=N^{[M]}_{\leq k}(r, \nu_{f}),\quad N^{[M]}_{f, >
k}(r)=N^{[M]}_{> k}(r, \nu_{f}).$$ In addition, if $M=\infty,$ we will omit
the superscript $M$ for brevity. On the other hand we have the following Jensen's formula:
$$N_{f}(r)-N_{\frac{1}{f}}(r)=\int_{S_{m}(r)}\log|f|\eta_{m}-\int_{S_{m}(1)}\log|f|\eta_{m}.$$
\par

\subsection{}
Let $Q$ be a hypersurface in $\mathbb{P}^{n}(\mathbb{C})$ of the degree $d$ defined by

$$\sum_{I\in \mathcal{I}_{d}}a_{I}x^{I}=0,$$
where $\mathcal{I}_{d}=\{(i_{0}, \ldots, i_{n})\in\mathbb{N}_{0}^{n+1}; i_{0}+\cdots+i_{n}=d\},$ $I=(i_{0}, \ldots, i_{n})\in \mathcal{I}_{d},$ $x^{I}=x_{0}^{i_{0}}\cdots x_{n}^{i_{n}}$ and $(c_{0}: \cdots: c_{n})$ is homogeneous coordinates of $\mathbb{P}^{n}(\mathbb{C}).$\par

Let $f: \mathbb{C}^{m}\rightarrow V\subset \mathbb{P}^{n}(\mathbb{C})$ be an algebraically nondegenerate meromorphic mapping into $V$ with a reduced representation $f=(f_{0}: \cdots : f_{n}).$ We define $$Q(f)=\sum_{I\in\mathcal{I}_{d}}a_{I}f^{I},$$
where $f^{I}=f_{0}^{i_{0}}\cdots f_{n}^{i_{n}}$ for $I=(i_{0}, \ldots, i_{n}).$ We denote by $f^{*}Q=\nu_{Q(f)}$ as divisor.

We define Nevanlinna's deficiency $\delta_{f}(Q)$ by
$$\delta_{f}(Q)=1-\limsup_{r\rightarrow\infty}\frac{N_{Q(f)}(r)}{T_{f}(r)}.$$
If $\delta_{f}(Q)>0,$ then $Q$ is called a deficient hypersurface in the sense of Nevanlinna.\par

As usual, $"\|P"$ means the assertion $P$ holds for all $r\in[0,\infty)$ excluding a Borel subset $E$ of the interval $[0,\infty)$ with $\int_{E}dr<\infty.$\par

\begin{theorem} \cite{quang-an-2014}\label{T21}(The Second Main Theorem for fixed hypersurfaces in subgeneral position) Let $V$ be a complex projective subvariety of $\mathbb{P}^{n}(\mathbb{C})$ of dimension $k(k\leq n).$ Let $\{Q_{i}\}_{i=1}^{q}$ be hypersurfaces of $\mathbb{P}^{n}(\mathbb{C})$ in $N-$subgeneral position with respect to $V,$ with $\deg Q_{i}=d_{i}$ $(1\leq i\leq q).$ Let $d$ be the least common multiple of $d_{i}$'s, i.e., $d=lcm(d_{1}, \ldots, d_{q}).$ Let $f$ be a meromorphic mapping of $\mathbb{C}^{m}$ into $V$ which is nondegenerate over $I_{d}(V).$ If $q>\frac{(2N-k+1)H_{V}(d)}{k+1},$ then we have

$$\|\left(q-\frac{(2N-k+1)H_{V}(d)}{k+1}\right)T_{f}(r)\leq \sum_{i=1}^{q}\frac{1}{d_{i}}N^{[H_{V}(d)-1]}_{Q_{i}(f)}(r)+o(T_{f}(r)).$$
\end{theorem}

\subsection{}
We denote by $\mathcal{M}$ (resp. $\mathcal{K}_{f}$) the field of all meromorphic functions (resp. small meromorphic functions) on $\mathbb{C}^{m}.$ Denote by $\mathcal{H}_{\mathbb{C}^{m}}$ the ring of all holomorphic functions on $\mathbb{C}^{m}.$ Let $Q$ be a homogeneous polynomial in $\mathcal{H}_{\mathbb{C}^{m}}[x_{0},\ldots, x_{n}]$ of degree $d\geq 1.$ Denote by $Q(z)$ be the homogeneous polynomial over $\mathbb{C}$ obtained by substituting a specific point $z\in\mathbb{C}^{m}$ into the coefficients of $Q.$ We also call a moving hyper surface in $\mathbb{P}^{n}(\mathbb{C})$ each homogeneous polynomial $Q\in\mathcal{H}_{\mathbb{C}^{m}}[x_{0},\ldots, x_{n}]$ such that the common zero set of all coefficients of $Q$ has codimension at least two.\par

Let $Q$ be a moving hypersurface in $\mathbb{P}^{n}(\mathbb{C})$ of degree $d$ given by
$$Q(z)=\sum_{I\in \mathcal{I}_{d}}a_{I}w^{I},$$
where $\mathcal{I}_{d}=\{(i_{0}, \ldots, i_{n})\in\mathbb{N}_{0}^{n+1}; i_{0}+\cdots+i_{n}=d\},$ $a_{I}\in \mathcal{H}_{\mathbb{C}^{m}}$ and $w^{I}=w^{i_{0}}\cdots w^{i_{n}}.$ We consider the meromorphic mapping $Q^{'}: \mathbb{C}^{m}\rightarrow\mathbb{P}^{N}(\mathbb{C}),$ where $N=\left(  \begin{array}{c}   n+d \\    n \\    \end{array}   \right)-1,$ given by
$$Q^{'}(z)=(a_{I_{0}}(z): \cdots: a_{I_{N}}(z))\quad \quad (\mathcal{I}_{d}=\{I_{0, \ldots, I_{N}}\}).$$ The moving hypersurface $Q$ is said to be "slowly" (with respect to $f$) if $\| T_{Q^{'}}(r)=o(T_{f}(r)).$ This is equivalent to $\| T_{\frac{a_{I_{i}}}{a_{I_{j}}}}(r)=o(T_{f}(r))$ for every $a_{I_{j}}\not\equiv 0.$\par

Let $\{Q_{i}\}_{i=1}^{q}$ be a family of moving hypersurfaces in $\mathbb{P}^{n}(\mathbb{C}),$ $\deg Q_{i}=d_{i}.$ Assume that
$$Q_{i}=\sum_{I\in \mathcal{I}_{d_{i}}}a_{iI}w^{I}.$$
We denote by $\tilde{\mathcal{K}}_{\{Q_{j}\}_{j=1}^{q}}$ the smallest subfield of $\mathcal{M}$ which contains $\mathbb{C}$ and all $\frac{a_{iI}}{a_{iJ}}$ with $a_{iJ}\not\equiv 0.$ We say that $\{Q_{i}\}_{i=1}^{q}$ are in weakly general position if there exits $z\in\mathbb{C}^{m}$ such that all $a_{iI}$ $(1\leq i\leq q, I\in\mathcal{I})$ are holomorphic at $z$ and for any $1\leq i_{0}<\cdots<i_{n}\leq q$ the system of equations
$$\left\{
 \begin{array}{c}
                                                                                                                    Q_{i_{j}}(z)(w_{0}, \ldots, w_{n})=0 \\
                                                                                                                    1\leq j\leq n \\
                                                                                                                  \end{array}\right.$$
has only the trivial solution $w=(0,\ldots, 0)$ in $\mathbb{C}^{n+1}.$\par

\begin{theorem} \cite{quang-3} \label{T22} (The Second Main Theorem for moving hypersurfaces in weakly general position) Let $f$ be a meromorphic mapping of $\mathbb{C}^{m}$ into $\mathbb{P}^{n}(\mathbb{C}).$ Let $Q_{i}$$(i=1,\ldots, q)$ be slowly (with respect to $f$) moving hypersurfaces of  $\mathbb{P}^{n}(\mathbb{C})$ in weakly general position with $\deg Q_{i}=d_{i}.$ Put $d$ be the least common multiple of $d_{1}, \ldots, d_{n+2}$, i. e., $d=lcm(d_{1}, \ldots, d_{n+2}),$ and $N=\left(  \begin{array}{c}   n+d \\   n \\   \end{array}   \right)-1.$ \par
(i) If $Q_{i}(f)\not\equiv 0$ $(1\leq i\leq q)$ and $q\geq nN+n+1,$ then we have
$$\| \frac{q}{nN+n+1}T_{f}(r)\leq \sum_{i=1}^{q}\frac{1}{d_{i}}N^{[N]}_{Q_{i}(f)}(r)+o(T_{f}(r)).$$\par

(ii) If $f$ is algebraically nondegenerate over $\tilde{\mathcal{K}}_{\{Q_{j}\}_{j=1}^{q}}$ and $q\geq N+2,$ then we have
$$\| \frac{q}{N+2}T_{f}(r)\leq \sum_{i=1}^{q}\frac{1}{d_{i}}N^{[N]}_{Q_{i}(f)}(r)+o(T_{f}(r)).$$
\end{theorem}

\subsection{Two lemmas}
\begin{lemma}\label{L1} Let $f:\mathbb{C}^{m}\rightarrow\mathbb{P}^{n}(\mathbb{C})$ be a meromorphic mapping which is nondegenerate over $I_{d}(V),$ and $\{Q_{j}\}_{j=1}^{q}$ be a family of hypersurfaces of $\mathbb{P}^{n}(\mathbb{C})$ in $N-$subgeneral position with respect to $V,$ with degree $\deg Q_{j}=d_{j}$ $(1\leq j\leq q).$ Let $d$ be the least common multiple of $d_{1}, \ldots, d_{q},$ that is $d=lcm(d_{1}, \ldots, d_{q}).$ Then
\begin{eqnarray*}\| &&\left[q-\frac{(2N-k+1)H_{V}(d)}{k+1}-\frac{H_{V}(d)-1}{d_{1}}\left(\frac{1}{m_{2}+1}-\frac{1}{m_{1}+1}\right)\left(1-\delta_{f}(Q_{1})\right)\right.\\
&&\left.-\sum_{j=1}^{q}\frac{1}{d_{j}}\frac{H_{V}(d)-1}{m_{j}+1}(1-\delta_{f}(Q_{j}))\right]T_{f}(r)
\\&\leq& \sum_{j=1}^{q}\frac{1}{d_{j}}\left(1-\frac{H_{V}(d)-1}{m_{2}+1}\right)N^{[H_{V}(d)-1]}_{Q_{j}(f), \leq m_{j}}(r)+o(T_{f}(r)),\end{eqnarray*}
where $m_{1}\geq m_{2}\geq \cdots \geq m_{q}$ are integers or $\infty.$
\end{lemma}

\begin{proof}
By the Second Main Theorem (\ref{T21}), we have
\begin{eqnarray*}\| &&\left[q-\frac{(2N-k+1)H_{V}(d)}{k+1}\right]T_{f}(r)\\
&\leq& \sum_{j=1}^{q}\frac{1}{d_{j}}N^{[H_{V}(d)-1]}_{Q_{j}(f)}(r)+o(T_{f}(r))\\
&\leq& \sum_{j=1}^{q}\frac{1}{d_{j}}\left[N^{[H_{V}(d)-1]}_{Q_{j}(f), \leq m_{j}}(r)+N^{[H_{V}(d)-1]}_{Q_{j}(f), > m_{j}}(r)\right]+o(T_{f}(r))\\
&\leq& \sum_{j=1}^{q}\frac{1}{d_{j}}\left[N^{[H_{V}(d)-1]}_{Q_{j}(f), \leq m_{j}}(r)+\frac{H_{V}(d)-1}{m_{j}+1}N_{Q_{j}(f), > m_{j}}(r)\right]+o(T_{f}(r))\\
&\leq& \sum_{j=1}^{q}\frac{1}{d_{j}}\left[N^{[H_{V}(d)-1]}_{Q_{j}(f), \leq m_{j}}(r)+\frac{H_{V}(d)-1}{m_{j}+1}\left(N_{Q_{j}(f)}(r)-N^{[H_{V}(d)-1]}_{Q_{j}(f), \leq m_{j}}(r)\right)\right]+o(T_{f}(r))\\
&\leq& \sum_{j=1}^{q}\frac{1}{d_{j}}\left[N^{[H_{V}(d)-1]}_{Q_{j}(f), \leq m_{j}}(r)+\frac{H_{V}(d)-1}{m_{j}+1}\left(N_{Q_{j}(f)}(r)-N^{[H_{V}(d)-1]}_{Q_{j}(f), \leq m_{j}}(r)\right)\right]+o(T_{f}(r)).\end{eqnarray*}
 Noting that $N_{Q_{j}(f)}(r)\leq (1-\delta_{f}(Q_{j}))T_{f}(r),$ we get from the above inequality that
\begin{eqnarray*}\| &&\left[q-\frac{(2N-k+1)H_{V}(d)}{k+1}\right]T_{f}(r)\\
&\leq&\sum_{j=1}^{q}\frac{1}{d_{j}}\left(1-\frac{H_{V}(d)-1}{m_{j}+1}\right)N^{[H_{V}(d)-1]}_{Q_{j}(f), \leq m_{j}}(r)\\&&+\sum_{j=1}^{q}\frac{1}{d_{j}}\frac{H_{V}(d)-1}{m_{j}+1}(1-\delta_{f}(Q_{j}))T_{f}(r)+o(T_{f}(r)),\end{eqnarray*}
thus,
\begin{eqnarray*}\| &&\left[q-\frac{(2N-k+1)H_{V}(d)}{k+1}\right]T_{f}(r)\\
&\leq&\frac{1}{d_{1}}\left(1-\frac{H_{V}(d)-1}{m_{1}+1}\right)N^{[H_{V}(d)-1]}_{Q_{1}(f), \leq m_{1}}(r)+\sum_{j=2}^{q}\frac{1}{d_{j}}\left(1-\frac{H_{V}(d)-1}{m_{2}+1}\right)N^{[H_{V}(d)-1]}_{Q_{j}(f), \leq m_{j}}(r)\\&&+\sum_{j=1}^{q}\frac{1}{d_{j}}\frac{H_{V}(d)-1}{m_{j}+1}(1-\delta_{f}(Q_{j}))T_{f}(r)+o(T_{f}(r))\\
&\leq&\frac{H_{V}(d)-1}{d_{1}}\left(\frac{1}{m_{2}+1}-\frac{1}{m_{1}+1}\right)N^{[H_{V}(d)-1]}_{Q_{1}(f), \leq m_{1}}(r)\\&&+\sum_{j=1}^{q}\frac{1}{d_{j}}\left(1-\frac{H_{V}(d)-1}{m_{2}+1}\right)N^{[H_{V}(d)-1]}_{Q_{j}(f), \leq m_{j}}(r)
\\&&+\sum_{j=1}^{q}\frac{1}{d_{j}}\frac{H_{V}(d)-1}{m_{j}+1}(1-\delta_{f}(Q_{j}))T_{f}(r)+o(T_{f}(r)).\end{eqnarray*}

Together with $N^{[H_{V}(d)-1]}_{Q_{1}(f), \leq m_{1}}(r)\leq (1-\delta_{f}(Q_{1}))T_{f}(r),$ the above inequality implies
\begin{eqnarray*}\| &&\left[q-\frac{(2N-k+1)H_{V}(d)}{k+1}\right]T_{f}(r)\\
&\leq&\frac{H_{V}(d)-1}{d_{1}}\left(\frac{1}{m_{2}+1}-\frac{1}{m_{1}+1}\right)(1-\delta_{f}(Q_{1}))T_{f}(r).\\
&&+\sum_{j=1}^{q}\frac{1}{d_{j}}\left(1-\frac{H_{V}(d)-1}{m_{2}+1}\right)N^{[H_{V}(d)-1]}_{Q_{j}(f), \leq m_{j}}(r)
\\&&+\sum_{j=1}^{q}\frac{1}{d_{j}}\frac{H_{V}(d)-1}{m_{j}+1}(1-\delta_{f}(Q_{j}))T_{f}(r)+o(T_{f}(r)).
\end{eqnarray*}
Hence, the lemma is proved.
\end{proof}

By a similar discussion as in proof of Lemma \ref{L1} using Theorem \ref{T22} instead of Theorem \ref{T21}, one can obtain the  following lemma.\par

\begin{lemma}\label{L2} Let $f$ be a meromorphic mapping of $\mathbb{C}^{m}$ into $\mathbb{P}^{n}(\mathbb{C}).$ Let $Q_{i}$$(i=1,\ldots, q)$ be slowly (with respect to $f$) moving hypersurfaces of  $\mathbb{P}^{n}(\mathbb{C})$ in weakly general position with $\deg Q_{i}=d_{i}.$ Put $d$ be the least common multiple of $d_{1}, \ldots, d_{n+2}$, i. e., $d=lcm(d_{1}, \ldots, d_{n+2}),$ and $N=\left(   \begin{array}{c}    n+d \\     n \\     \end{array}  \right)-1.$ \par
 Then
\begin{eqnarray*}\| &&\left[\frac{q}{p}-\frac{N}{d_{1}}\left(\frac{1}{m_{2}+1}-\frac{1}{m_{1}+1}\right)\left(1-\delta_{f}(Q_{1})\right)\right.\\
&&\left.-\sum_{j=1}^{q}\frac{1}{d_{j}}\frac{N}{m_{j}+1}(1-\delta_{f}(Q_{j}))\right]T_{f}(r)
\\&\leq& \sum_{j=1}^{q}\frac{1}{d_{j}}\left(1-\frac{N}{m_{2}+1}\right)N^{[N]}_{Q_{j}(f), \leq m_{j}}(r)+o(T_{f}(r)),\end{eqnarray*}
 where $$p:=\left\{
           \begin{array}{ll}
             nN+n+1, & \hbox{if $Q_{j}(f)\not\equiv 0$ and $q\geq nN+n+1$;} \\
             N+2, & \hbox{if $f$ is algebraically nondegenerate over $\tilde{\mathcal{K}}_{\{Q_{j}\}_{j=1}^{q}}$ and $q\geq N+2,$}
           \end{array}
         \right.
 $$
 and
 $m_{1}\geq m_{2}\geq \cdots \geq m_{q}$ are integers or $\infty.$\end{lemma}

\section{Proof of Theorem \ref{T1}}

Suppose that $f$ and $g$ have reduced representations respectively as follows:
$$f=(f_{0}: f_{1}: \cdots: f_{n})\quad \mbox{and}\quad g=(g_{0}: g_{1}: \cdots: g_{n}).$$ Assume that $f\neq g.$ Then there two distinct indices $s$ and $t$ in $\{0, 1, \ldots, n\}$ such that $$H:=f_{s}g_{t}-f_{t}g_{s}\not\equiv 0.$$ By the assumptions (a) and (b) of the theorem, we get that any $z$ in $$\cup_{j=1}^{q}{\{z\in\mathbb{C}^{m}:0<\nu_{Q_{j}(f)}(z)\leq m_{j}}\}$$ must be a zero of $H,$ and further more,
$$\nu_{H}(z)\geq \sum_{j=1}^{q}\min\{\nu_{Q_{j}(f), \leq m_{j}}(z),1\}$$ outside an analytic subset of codimension at least two, which implies that
$$N_{H}(r)\geq \sum_{j=1}^{q}N_{Q_{j}(f), \leq m_{j}}^{[1]}(r).$$

On the other hand, on can get from the definition of the characteristic function and Jensen formula that
\begin{eqnarray*}N_{H}(r)&=&\int_{S_{m}(r)}\log|f_{s}g_{t}-f_{t}g_{s}|\sigma_{m}\\
&\leq& \int_{S_{m}(r)}\log||f||\sigma_{m}+\int_{S_{m}(r)}\log||g||\sigma_{m}\\
&=&T_{f}(r)+T_{g}(r).
\end{eqnarray*} Therefore, we obtain from two inequalities above that
$$\sum_{j=1}^{q}N_{Q_{j}(f), \leq m_{j}}^{[1]}(r)\leq T_{f}(r)+T_{g}(r).$$
Similar discussion for $g$ we have $$\sum_{j=1}^{q}N_{Q_{j}(g), \leq m_{j}}^{[1]}(r)\leq T_{f}(r)+T_{g}(r).$$ Hence,
\begin{eqnarray*} 2(T_{f}(r)+T_{g}(r))&\geq& \sum_{j=1}^{q}N_{Q_{j}(f), \leq m_{j}}^{[1]}(r)+\sum_{j=1}^{q}N_{Q_{j}(g), \leq m_{j}}^{[1]}(r)\\
&=& \sum_{j=1}^{q}N_{Q_{j}^{\frac{d}{d_{j}}}(f), \leq m_{j}}^{[1]}(r)+\sum_{j=1}^{q}N_{Q_{j}^{\frac{d}{d_{j}}}(g), \leq m_{j}}^{[1]}(r).
\end{eqnarray*}

By Lemma \ref{L1} for $f$ and the hypersurfaces $Q_{j}^{\frac{d}{d_{j}}}$ of the common degree $d,$ we have
\begin{eqnarray*}\| &&\left[q-\frac{(2N-k+1)H_{V}(d)}{k+1}-\frac{H_{V}(d)-1}{d}\left(\frac{1}{m_{2}+1}-\frac{1}{m_{1}+1}\right)\right.\\
&&\left.-\sum_{j=1}^{q}\frac{1}{d}\frac{H_{V}(d)-1}{m_{j}+1}\right]T_{f}(r)
\\&\leq& \sum_{j=1}^{q}\frac{1}{d}\left(1-\frac{H_{V}(d)-1}{m_{2}+1}\right)N^{[H_{V}(d)-1]}_{Q_{j}^{\frac{d}{d_{j}}}(f), \leq m_{j}}(r)+o(T_{f}(r)),\end{eqnarray*}
which means
\begin{eqnarray*}\| &&\left[qd-\frac{d(2N-k+1)H_{V}(d)}{k+1}-(H_{V}(d)-1)\left(\frac{1}{m_{2}+1}-\frac{1}{m_{1}+1}\right)\right.\\
&&\left.-\sum_{j=1}^{q}\frac{H_{V}(d)-1}{m_{j}+1}\right]T_{f}(r)
\\&\leq& \sum_{j=1}^{q}\left(1-\frac{H_{V}(d)-1}{m_{2}+1}\right)N^{[H_{V}(d)-1]}_{Q_{j}^{\frac{d}{d_{j}}}(f), \leq m_{j}}(r)+o(T_{f}(r))\\
&\leq& (H_{V}(d)-1)\left(1-\frac{H_{V}(d)-1}{m_{2}+1}\right)\sum_{j=1}^{q}N^{[1]}_{Q_{j}^{\frac{d}{d_{j}}}(f), \leq m_{j}}(r)+o(T_{f}(r)).\end{eqnarray*}
Similarly, for $g$ we also have
\begin{eqnarray*}\| &&\left[qd-\frac{d(2N-k+1)H_{V}(d)}{k+1}-(H_{V}(d)-1)\left(\frac{1}{m_{2}+1}-\frac{1}{m_{1}+1}\right)\right.\\
&&\left.-\sum_{j=1}^{q}\frac{H_{V}(d)-1}{m_{j}+1}\right]T_{g}(r)\\
&\leq& (H_{V}(d)-1)\left(1-\frac{H_{V}(d)-1}{m_{2}+1}\right)\sum_{j=1}^{q}N^{[1]}_{Q_{j}^{\frac{d}{d_{j}}}(g), \leq m_{j}}(r)+o(T_{g}(r)).\end{eqnarray*}\par

Henceforth,  we get form the inequalities above that
\begin{eqnarray*}\| &&\left[qd-\frac{d(2N-k+1)H_{V}(d)}{k+1}-(H_{V}(d)-1)\left(\frac{1}{m_{2}+1}-\frac{1}{m_{1}+1}\right)\right.\\
&&\left.-\sum_{j=1}^{q}\frac{H_{V}(d)-1}{m_{j}+1}\right](T_{f}(r)+T_{g}(r))\\
&\leq& 2(H_{V}(d)-1)\left(1-\frac{H_{V}(d)-1}{m_{2}+1}\right)(T_{f}(r)+T_{g}(r))+o(T_{f}(r)+T_{g}(r)),\end{eqnarray*}\par
which implies
\begin{eqnarray*}\| &&\sum_{j=3}^{q}\frac{m_{j}}{m_{j}+1}(T_{f}(r)+T_{g}(r))\\
&\leq& \left(q+\frac{d(2N-k+1)H_{V}(d)}{(k+1)(H_{V}(d)-1)}-\frac{qd}{H_{V}(d)-1}-\frac{4-2H_{V}(d)}{m_{2}+1})\right)(T_{f}(r)+T_{g}(r))\\
&&+o((T_{f}(r)+T_{g}(r))).
\end{eqnarray*}This is a contradiction. Therefore, we completely prove the theorem.

\section{Proof of Theorem \ref{T2}}
The proof of this theorem is some similar as the proof of Theorem \ref{T1}. Assume that $f$ and $g$ have reduced representations respectively as follows:
$$f=(f_{0}: f_{1}: \cdots: f_{n})\quad \mbox{and}\quad g=(g_{0}: g_{1}: \cdots: g_{n}).$$ Then by Lemma \ref{L2} for $f$ and the hypersurfaces $Q_{j}^{\frac{d}{d_{j}}}$ of the common degree $d,$ we have
\begin{eqnarray*}\| &&\left[\frac{q}{p}-\frac{N}{d}\left(\frac{1}{m_{2}+1}-\frac{1}{m_{1}+1}\right)-\sum_{j=1}^{q}\frac{1}{d}\frac{N}{m_{j}+1}\right]T_{f}(r)
\\&\leq& \sum_{j=1}^{q}\frac{1}{d}\left(1-\frac{N}{m_{2}+1}\right)N^{[N]}_{Q_{j}^{\frac{d}{d_{j}}}(f), \leq m_{j}}(r)+o(T_{f}(r)),\end{eqnarray*} which means that
\begin{eqnarray*}\| &&\left[\frac{qd}{p}-N\left(\frac{1}{m_{2}+1}-\frac{1}{m_{1}+1}\right)-\sum_{j=1}^{q}\frac{N}{m_{j}+1}\right]T_{f}(r)
\\&\leq& \left(1-\frac{N}{m_{2}+1}\right)\sum_{j=1}^{q}N^{[N]}_{Q_{j}^{\frac{d}{d_{j}}}(f), \leq m_{j}}(r)+o(T_{f}(r)).\end{eqnarray*} Similarly for $g,$
\begin{eqnarray*}\| &&\left[\frac{qd}{p}-N\left(\frac{1}{m_{2}+1}-\frac{1}{m_{1}+1}\right)-\sum_{j=1}^{q}\frac{N}{m_{j}+1}\right]T_{g}(r)
\\&\leq& \left(1-\frac{N}{m_{2}+1}\right)\sum_{j=1}^{q}N^{[N]}_{Q_{j}^{\frac{d}{d_{j}}}(g), \leq m_{j}}(r)+o(T_{g}(r)).\end{eqnarray*}Hence,
\begin{eqnarray*}\| &&\left[\frac{qd}{p}-N\left(\frac{1}{m_{2}+1}-\frac{1}{m_{1}+1}\right)-\sum_{j=1}^{q}\frac{N}{m_{j}+1}\right](T_{f}(r)+T_{g}(r))
\\&\leq& \left(1-\frac{N}{m_{2}+1}\right)\sum_{j=1}^{q}\left(N^{[N]}_{Q_{j}^{\frac{d}{d_{j}}}(f), \leq m_{j}}(r)+N^{[N]}_{Q_{j}^{\frac{d}{d_{j}}}(g), \leq m_{j}}(r)\right)+o(T_{f}(r)+T_{g}(r)),\end{eqnarray*} where $$p:=\left\{
           \begin{array}{ll}
             nN+n+1, & \hbox{if $Q_{j}(f)\not\equiv 0$ and $q\geq nN+n+1$;} \\
             N+2, & \hbox{if $f$ is algebraically nondegenerate over $\tilde{\mathcal{K}}_{\{Q_{j}\}_{j=1}^{q}}$ and $q\geq N+2,$}
           \end{array}
         \right.
 $$
\par

Assume that $f\neq g.$ Then there two distinct indices $s$ and $t$ in $\{0, 1, \ldots, n\}$ such that $$H:=f_{s}g_{t}-f_{t}g_{s}\not\equiv 0.$$ Set $S=\cup\{\cap_{j=0}^{k}\nu_{Q_{i_{j}}(f)}; 1\leq i_{0}<\ldots<i_{k}\leq q\}.$ Then $S$ is either an analytic subset of codimension at least two of $\mathbb{C}^{m}$ or an empty set. By the assumptions (a) and (b) of the theorem, we get that any $z$ in $$\cup_{j=1}^{q}{\{z\in\mathbb{C}^{m}:0<\nu_{Q_{j}(f)}(z)\leq m_{j}}\}\setminus S$$  must be a zero of $H,$ and further more,
$$\nu_{H}(z)\geq \frac{1}{k}\sum_{j=1}^{q}\min\{\nu_{Q_{j}(f), \leq m_{j}}(z),1\}$$ outside the analytic subset $S$ of codimension at least two, which implies that
$$N_{H}(r)\geq \frac{1}{k}\sum_{j=1}^{q}N_{Q_{j}(f), \leq m_{j}}^{[1]}(r).$$

On the other hand, we can get also from the definition of the characteristic function and Jensen formula that
\begin{eqnarray*}N_{H}(r)&=&\int_{S_{m}(r)}\log|f_{s}g_{t}-f_{t}g_{s}|\sigma_{m}\\
&\leq& \int_{S_{m}(r)}\log||f||\sigma_{m}+\int_{S_{m}(r)}\log||g||\sigma_{m}\\
&=&T_{f}(r)+T_{g}(r).
\end{eqnarray*} Therefore, we obtain from two inequalities above that
$$\frac{1}{k}\sum_{j=1}^{q}N_{Q_{j}(f), \leq m_{j}}^{[1]}(r)\leq T_{f}(r)+T_{g}(r).$$
Similar discussion for $g$ we have $$\frac{1}{k}\sum_{j=1}^{q}N_{Q_{j}(g), \leq m_{j}}^{[1]}(r)\leq T_{f}(r)+T_{g}(r).$$ Hence, we get that
\begin{eqnarray*} 2(T_{f}(r)+T_{g}(r))&\geq& \frac{1}{k}\sum_{j=1}^{q}N_{Q_{j}(f), \leq m_{j}}^{[1]}(r)+\frac{1}{k}\sum_{j=1}^{q}N_{Q_{j}(g), \leq m_{j}}^{[1]}(r)\\
&=& \frac{1}{k}\sum_{j=1}^{q}N_{Q_{j}^{\frac{d}{d_{j}}}(f), \leq m_{j}}^{[1]}(r)+\frac{1}{k}\sum_{j=1}^{q}N_{Q_{j}^{\frac{d}{d_{j}}}(g), \leq m_{j}}^{[1]}(r)\\
&\geq& \frac{1}{kN}\sum_{j=1}^{q}\left(N_{Q_{j}^{\frac{d}{d_{j}}}(f), \leq m_{j}}^{[N]}(r)+N_{Q_{j}^{\frac{d}{d_{j}}}(g), \leq m_{j}}^{[N]}(r)\right).\\
\end{eqnarray*}

Henceforth, we get that
\begin{eqnarray*}\| &&\left[\frac{qd}{p}-N\left(\frac{1}{m_{2}+1}-\frac{1}{m_{1}+1}\right)-\sum_{j=1}^{q}\frac{N}{m_{j}+1}\right](T_{f}(r)+T_{g}(r))\\
&\leq& 2kN\left(1-\frac{N}{m_{2}+1}\right)(T_{f}(r)+T_{g}(r))+o(T_{f}(r)+T_{g}(r)),\end{eqnarray*}\par
which implies
\begin{eqnarray*}\| &&\sum_{j=3}^{q}\frac{m_{j}}{m_{j}+1}(T_{f}(r)+T_{g}(r))\\
&\leq& \left(q+2k-2-\frac{qd}{pN}-\frac{2kN-2}{m_{2}+1})\right)(T_{f}(r)+T_{g}(r))\\
&&+o((T_{f}(r)+T_{g}(r))).
\end{eqnarray*} This is a contradiction. Therefore, the theorem is completely proved.

\end{document}